\documentclass[a4paper,11pt]{article}

\usepackage{amsmath}
\usepackage{amssymb}
\usepackage{amsthm}
\usepackage{graphicx}
\usepackage{amscd}
\usepackage{epic, eepic}
\usepackage{url}
\usepackage{color}
\usepackage[utf8]{inputenc} 
\usepackage{comment}
\usepackage{enumerate}   
\usepackage{todonotes}
\usepackage{graphicx}
\usepackage{epstopdf}
\usepackage{enumitem}
\usepackage{tikz}
\usepackage[top=25mm, bottom=25mm, left=28mm, right=28mm]{geometry}
\usepackage[bookmarks=false,hidelinks]{hyperref}
\usepackage{bbm}

\makeatletter
\renewenvironment{proof}[1][\proofname] {\par\pushQED{\qed}\normalfont\topsep6\p@\@plus6\p@\relax\trivlist\item[\hskip\labelsep\bfseries#1\@addpunct{.}]\ignorespaces}{\popQED\endtrivlist\@endpefalse}
\makeatother

\newtheorem{theorem}{\bf Theorem}[section]
\newtheorem{lemma}[theorem]{\bf Lemma}

\theoremstyle{definition}

\newtheorem{remark}[theorem]{\bf Remark}
\newtheorem{definition}[theorem]{\bf Definition}

\def\eps{\varepsilon}
\def\vb{\textbf{b}}
\def\vv{\textbf{v}}

\def\vA{\textbf{A}}

\def\cG{\mathcal{G}}
\def\cH{\mathcal{H}}

\def\ex{\mathrm{ex}}

\title{}
\date{}

\author{
Domagoj Brada\v{c}\thanks{Department of Mathematics, ETH, Z\"urich, Switzerland. Research supported in part by SNSF grant 200021\_196965. Email: \textbf{\{domagoj.bradac, lior.gishboliner, benjamin.sudakov\}@math.ethz.ch}.}
\and 
Lior Gishboliner\footnotemark[1]
\and
Oliver Janzer\thanks{Department of Mathematics, ETH, Z\"urich, Switzerland. Research supported by an ETH Z\"urich Postdoctoral Fellowship 20-1 FEL-35. Email: \textbf{oliver.janzer@math.ethz.ch}.}
\and Benny Sudakov\footnotemark[1]
}

\title{Asymptotics of the hypergraph bipartite Tur\'an problem}

\begin{document}

\maketitle

\begin{abstract}
    For positive integers $s,t,r$, let $K_{s,t}^{(r)}$ denote the $r$-uniform hypergraph whose vertex set is the union of pairwise disjoint sets $X,Y_1,\dots,Y_t$, where $|X| = s$ and $|Y_1| = \dots = |Y_t| = r-1$, and whose edge set is $\{\{x\} \cup Y_i: x \in X, 1\leq i\leq t\}$. The study of the Tur\'an function of $K_{s,t}^{(r)}$ received considerable interest in recent years. Our main results are as follows. First, we show \nolinebreak that
    \begin{equation}\label{eq:upper bound}
        \ex(n,K_{s,t}^{(r)}) = O_{s,r}(t^{\frac{1}{s-1}}n^{r - \frac{1}{s-1}})
    \end{equation}
    for all $s,t\geq 2$ and $r\geq 3$, improving the power of $n$ in the previously best bound and resolving a question of Mubayi and Verstra\"ete about the dependence of $\ex(n,K_{2,t}^{(3)})$ on $t$. Second, we show that \eqref{eq:upper bound} is tight when $r$ is even and $t \gg s$. This disproves a conjecture of Xu, Zhang and Ge. Third, we show that \eqref{eq:upper bound} is {\em not} tight for $r = 3$, namely that $\ex(n,K_{s,t}^{(3)}) = O_{s,t}(n^{3 - \frac{1}{s-1} - \varepsilon_s})$ (for all $s\geq 3$). This indicates that the behaviour of $\ex(n,K_{s,t}^{(r)})$ might depend on the parity of $r$. Lastly, we prove a conjecture of Ergemlidze, Jiang and Methuku on the hypergraph analogue of the bipartite Tur\'an problem for graphs with bounded degrees on one side. Our tools include a novel twist on the dependent random choice method as well as a variant of the celebrated norm graphs constructed by Koll{\'a}r, R{\'o}nyai and Szab{\'o}.
\end{abstract}

\section{Introduction}

Let $H$ be an $r$-uniform hypergraph. The Tur\'an function $\ex(n,H)$ of $H$ is the largest number of edges in an $r$-uniform hypergraph on $n$ vertices with no copy of $H$. The study of the function $\ex(n,H)$ for various hypergraphs $H$ is one of the central problems of extremal combinatorics. In the graph case $r=2$, the Tur\'an function is fairly well understood unless $H$ bipartite. On the other hand, for $r \ge 3$, our understanding of the Tur\'an function is much worse and there is only a small number of tight results. For example, determining the answer for the $3$-uniform clique on $4$ vertices is still open. Given the difficulty of the problem even for hypergraph cliques, various hypergraphs originating from graphs have been considered for which better bounds can be obtained. Mubayi \cite{mubayi2006hypergraph} studied a hypergraph extension of the graph clique and his result was refined by Pikhurko \cite{Pikh05}. A different hypergraph extension of the triangle was introduced by Frankl \cite{Frankl1990AsymptoticSO} who determined the asymptotics of its Tur\'an number and an exact answer was given by Keevash and Sudakov \cite{KS05}. Sidorenko \cite{Sidorenko89} asymptotically determined the Tur\'an number of a hypergraph extension of trees. We refer the interested reader to an extensive survey of Keevash \cite{keevash_2011} on Tur\'{a}n problems for non--$r$-partite $r$-uniform hypergraphs.

It is well-known that for $r$-partite $H$, one has $\ex(n,H) = O(n^{r - \varepsilon})$ for some $\varepsilon = \varepsilon(H) > 0$ and the main goal here is to determine or estimate the best possible $\varepsilon(H)$.
One of the very old such Tur\'an-type questions for hypergraphs is a problem of Erd\H{o}s \cite{Erdos77}, asking for the maximum number $f_r(n)$ of edges in an $r$-uniform hypergraph on $n$ vertices which does not have four distinct edges $A,B,C,D$ satisfying $A \cup B = C \cup D$ and $A \cap B = C \cap D = \emptyset$. Note that in this problem forbidden hypergraphs originate quite naturally from a four-cycle. 
Erd\H{o}s in particular asked whether $f_r(n) = O(n^{r-1})$. 
This was answered affirmatively by F\"uredi \cite{Fur84}, who showed that $f_r(n) \leq 3.5 \binom{n}{r-1}$. 
Mubayi and Verstra\"ete \cite{MV04} extended Erd\H{o}s's question by considering the following family of $r$-uniform hypergraphs which generalize complete bipartite graphs: 
for positive integers $r,s,t$, let $K_{s,t}^{(r)}$ denote the $r$-uniform hypergraph whose vertex set consists of disjoint sets $X,Y_1,\dots,Y_t$, where $|X| = s$ and $|Y_1| = \dots = |Y_t| = r-1$, and whose edge set is $\{\{x\} \cup Y_i: x \in X, 1\leq i\leq t\}$. Note that $K_{s,t}^{(r)}$ is $r$-partite and for $r = 2$, $K_{s,t}^{(2)}$ is just the $s \times t$ complete bipartite graph. Observe that the edges of $K^{(r)}_{2,2}$ form a configuration $A,B,C,D$ as in the definition of $f_r(n)$. Hence, $f_r(n) \leq \ex(n,K_{2,2}^{(r)})$ (this is in fact an equality for $r = 3$). Mubayi and Verstra\"ete \cite{MV04} proved that $\ex(n,K_{2,2}^{(r)})\leq 3\binom{n}{r-1}+O(n^{r-2})$, improving the constant in F\"uredi's result.  Pikhurko and Verstra\"ete \cite{PV09} improved the coefficient of $\binom{n}{r-1}$ further. It remains open whether $\ex(n,K_{2,2}^{(r)}) = (1 + o(1)) \binom{n-1}{r-1}$, as conjectured by F\"uredi. That $\binom{n-1}{r-1}$ is a lower bound can be seen by considering the star, i.e. the hypergraph consisting of all edges containing a fixed vertex. 

Mubayi and Verstra\"ete \cite{MV04} initiated the study of $\ex(n,K_{s,t}^{(3)})$ for general $s,t$, and proved that
$\ex(n,K_{s,t}^{(3)})\leq C_{s,t}n^{3-1/s}$ as well as that $\ex(n,K_{s,t}^{(3)})\geq c_t n^{3-2/s}$ for $t>(s-1)!$. For small values of $s$, they obtained more accurate estimates. Namely, for $s=3$, they improved their bound to $\ex(n,K_{3,t}^{(3)})\leq C_tn^{13/5}$, while
 for $s = 2$, they showed that $\ex(n,K_{2,t}^{(3)})\leq t^4 \binom{n}{2}$ and that $\ex(n,K_{2,t}^{(3)})\geq \frac{2t-1}{3}\binom{n}{2}$ for infinitely many $n$. They further asked to determine the correct dependence of $\ex(n,K_{2,t}^{(3)})$ on $t$. Ergemlidze, Jiang and Methuku \cite{EJM20} improved the upper bound to $\ex(n,K_{2,t}^{(3)})\leq (15t\log t+40t)n^2$, leaving a $\log t$ gap from the lower bound of $\Omega(tn^2)$. For $r > 3$, little is known. Ergemlidze, Jiang and Methuku found a construction showing $\ex(n, K_{2,t}^{(4)}) = \Omega(tn^3)$. Xu, Zhang and Ge \cite{XZG20,XZG21} proved a tight bound on $\ex(n,K_{s,t}^{(r)})$ when $s$ is much larger than $t$, using a standard application of the random algebraic method of Bukh \cite{Bukh15}. 

Our first result, Theorem \ref{thm:r-uniform bipartite}, achieves two goals. First, it resolves the problem of Mubayi and Verstra\"ete by proving that $\ex(n,K_{2,t}^{(3)})= \Theta(tn^2)$. And second, it improves the upper bound of \cite{MV04} on $\ex(n,K_{s,t}^{(3)})$  {\em for every} $s$ and $t$ by reducing the exponent of $n$ from $3 - \frac{1}{s}$ to $3 - \frac{1}{s-1}$. We also obtain an analogous result for every $r \geq 3$.  
The proof of this bound relies on a new weighted variant of the dependent random choice method (see \cite{Fox2011DependentRC} for a description of the technique and a brief history).

\begin{theorem} \label{thm:r-uniform bipartite}
	 For any $s,t\geq 2$ and $r \geq 3$ there is a constant $C_s$ depending only on $s$ such that $$\ex(n,K_{s,t}^{(r)})\leq C_{s} t^{\frac{1}{s-1}}n^{r-\frac{1}{s-1}}.$$
	 \noindent
	 In particular, $$\ex(n,K_{2,t}^{(3)})\leq Ctn^2$$ for some absolute constant $C$.
\end{theorem}

Our next result shows that, somewhat surprisingly, the bound in Theorem \ref{thm:r-uniform bipartite} is tight in terms of both $n$ and $t$ if the uniformity $r$ is even and $t \gg s$. Our construction uses as building blocks a variation of the norm graphs, introduced by Koll\'ar, R\'onyai and Szab\'o \cite{KRS96}, which might be of independent interest.

\begin{theorem} \label{thm:right dependence on t}
     For any positive integers $s\geq 2$ and $k$, there is a positive constant $c=c(k,s)$ such that for every integer $t>(s-1)!$, if $n$ is sufficiently large, then
     $$\ex(n,K_{s,t}^{(2k)})\geq ct^{\frac{1}{s-1}}n^{2k-\frac{1}{s-1}}.$$
\end{theorem}
By combining Theorems \ref{thm:r-uniform bipartite} and \ref{thm:right dependence on t}, we see that 
$\ex(n,K_{s,t}^{(r)}) = \Theta_{r,s}(t^{\frac{1}{s-1}}n^{r-\frac{1}{s-1}})$ if $r\geq 4$ is even and $t > (s-1)!$. (Here and elsewhere in the paper, $\Theta_{r,s}$ means that the implied constants can depend on $r$ and $s$.) In the special case $s = 2$, this gives $\ex(n,K_{2,t}^{(r)}) = \Theta_r(t n^{r-1})$ for even $r\geq 4$. This partially answers a question of Ergemlidze, Jiang and Methuku \cite{EJM20}, who asked to determine the dependence of $\ex(n,K_{2,t}^{(r)})$ on $t$.  
Also, Theorem \ref{thm:right dependence on t} disproves a conjecture of Xu, Zhang and Ge \cite[Conjecture 5.1]{XZG21} which stated that $\ex(n,K_{s,t}^{(r)}) = \Theta_{r,s,t}(n^{r-2/s})$ for all $2 \leq s \leq t$.  

It is natural to ask whether the bound in Theorem \ref{thm:r-uniform bipartite} is tight for odd uniformities as well. Our next theorem shows that this is {\em not} the case for $r = 3$. This indicates that, perhaps surprisingly, the parity of $r$ may play a role. 

\begin{theorem} \label{thm:eps improvement}
     For any $s\geq 3$, there exists some $\eps>0$ such that for any $t$, $$\ex(n,K_{s,t}^{(3)})\leq C_{s,t} n^{3-\frac{1}{s-1}-\eps}.$$
\end{theorem}

Theorems \ref{thm:right dependence on t} and \ref{thm:eps improvement} together show that $\ex(n,K_{s,t}^{(r)})/n^{r-1}$ has a different order of magnitude for even $r \geq 4$ and for $r = 3$. Indeed,
for even $r$ the function $\ex(n,K_{s,t}^{(r)})/n^{r-1}$ is asymptotically $\Theta_{r,s,t}(n^{1 - \frac{1}{s-1}})$, assuming $s \ll t$, while for $r = 3$ this function is smaller by at least a factor of $n^{\varepsilon}$. 
This runs contrary to a claim made in \cite[Proposition 1]{EJM20} that $\ex(n,K_{s,t}^{(r)}) \leq O_{r,s,t}(n^{r-3}) \cdot \ex(n,K_{s,t}^{(3)})$. One can check that the proof suggested in \cite{EJM20} is incorrect. Moreover, as we now see, the statement itself is disproved by Theorems \ref{thm:right dependence on t} and \nolinebreak \ref{thm:eps improvement}.

It would be very interesting to determine if Theorem \ref{thm:eps improvement} can be extended to all odd uniformities $r$. If so, then this would be a rare example of an extremal problem where the answer depends on the parity of the uniformity. (See \cite{KS05} for another hypergraph Tur\'an problem where the extremal construction depends heavily on number theoretic properties of the parameters.) The first open case is $r = 5$: is it true that $\ex(n,K_{s,t}^{(5)})=O(n^{5-\frac{1}{s-1}-\eps})?$ 

We end with some results for a more general family of hypergraphs. Let $G$ be a bipartite graph with an ordered bipartition $(X,Y)$, $Y = \{y_1,\dots,y_m\}$. Following \cite{EJM20}, we define $G_{X,Y}^{(r)}$ to be the $r$-uniform hypergraph whose vertex set consists of disjoint sets $X,Y_1,\dots,Y_m$, $|Y_1| = \dots = |Y_m| = r-1$, and whose edge set is $\{\{x\} \cup Y_i : \{x,y_i\} \in E(G)\}$. Note that if $G=K_{s,t}$ with $X$ being the part of size $s$ and $Y$ being the part of size $t$, then $G_{X,Y}^{(r)}=K_{s,t}^{(r)}$.
Ergemlidze, Jiang and Methuku \cite{EJM20} asked whether it is true that if all vertices in $Y$ have degree at most $2$ in $G$, then $\ex(n,G_{X,Y}^{(r)}) = O(n^{r-1})$ where the implied constant depends only on $G$ and $r$. Here we resolve this conjecture in greater generality.

\begin{theorem}\label{thm:degree-s on one side}
     Let $s\geq 2$, $r\geq 3$ and let $G$ be a bipartite graph with an ordered bipartition $(X,Y)$ such that every vertex in $Y$ has degree at most $s$. Then $\ex(n,G_{X,Y}^{(r)})\leq Cn^{r-\frac{1}{s-1}}$ where $C$ only depends on $G$ and $r$.
\end{theorem}

\noindent Note that Theorem \ref{thm:right dependence on t} shows that this bound can be attained whenever $r$ is even.

Finally, we consider the hypergraph $G_{X,Y}^{(r)}$ when $G=C_{2t}$ is the cycle of length $2t$. Let us write $C_{2t}^{(r)}$ for this hypergraph. In an unpublished work, Jiang and Liu showed that $\Omega_r(tn^{r-1}) \leq \ex(n,C_{2t}^{(r)}) \leq O_r(t^5n^{r-1})$. The lower bound is obtained by taking all edges which contain one of $t-1$ vertices. This hypergraph has cover number $t-1$, so it cannot contain $C_{2t}^{(r)}$, which has cover number $t$. Ergemlidze, Jiang and Methuku \cite{EJM20} improved the upper bound to $\ex(n,C_{2t}^{(r)}) \leq O_r(t^2(\log t) n^{r-1})$. We determine the correct dependence on $t$.
\begin{theorem}\label{thm:cycle}
     For every $t \geq 2$ and $r \geq 3$, we have that $\ex(n,C_{2t}^{(r)}) = \Theta_r(tn^{r-1})$.
\end{theorem}

The rest of the paper is organized as follows. In Section \ref{sec:upper}, we prove a general result which implies Theorems \ref{thm:r-uniform bipartite}, \ref{thm:degree-s on one side} and \ref{thm:cycle}. In Section \ref{sec:lower}, we prove Theorem \ref{thm:right dependence on t}. In Section \ref{sec:improved upper}, we prove Theorem \ref{thm:eps improvement}. In Section \ref{sec:concluding remarks}, we give some concluding remarks.

\section{Upper bounds} \label{sec:upper}

In what follows, for an $r$-uniform hypergraph $\cG$ and a set $S=\{v_1,\dots,v_{r-1}\}\subset V(\cG)$, we write $d_{\cG}(S)$ or $d_{\cG}(v_1,\dots,v_{r-1})$ for the number of vertices $v_r\in V(\cG)$ such that $v_1v_2\dots v_r\in E(\cG)$. We omit the subscript when the hypergraph is clear. Given a vertex $v \in V(\cG)$, the link hypergraph of $v$ (with respect to $\cG$) is the $(r-1)$-uniform hypergraph containing all $(r-1)$-sets which together with $v$ form an edge in $\cG$.

\begin{definition}
    In an $r$-uniform hypergraph $\cG$, we call a set $S\subset V(\cG)$ \emph{$t$-rich} if there are sets $T_1,T_2,\dots,T_t\subset V(\cG)$ of size $r-1$ such that $S,T_1,\dots,T_t$ are pairwise disjoint and $\{u\}\cup T_i\in E(\cG)$ for every $u\in S$ and $i\in [t]$.
\end{definition}

Note that if $\cG$ has any $t$-rich set of size $s$, then it contains $K_{s,t}^{(r)}$ as a subgraph.

\begin{theorem} \label{thm:dependent random choice}
    Let $\alpha>1$ be a real number and let $r\geq 3$, $s\geq 2$, $t$ and $n$ be positive integers. Then there is a constant $C$ which depends only on $s$ such that the following holds. If $\cG$ is an $n$-vertex $r$-uniform hypergraph with at least $C\alpha^{\frac{1}{s-1}}t^{\frac{1}{s-1}}n^{r-\frac{1}{s-1}}$ hyperedges, then there is a set $A\subset V(\cG)$ of size at least $\alpha^{\frac{1}{s-1}}t^{\frac{1}{s-1}}n^{1-\frac{1}{s-1}}$ (and at least $s$) such that the proportion of $t$-rich sets of size $s$ in $A$ is at least $1-\alpha^{-1}$.
\end{theorem}

Observe that the conclusion of this theorem implies that $\cG$ contains $K_{s,t}^{(r)}$ as a subgraph, so Theorem \ref{thm:r-uniform bipartite} follows immediately (by taking $\alpha=2$, for example). Moreover, as we will see shortly, Theorem \ref{thm:dependent random choice} also implies Theorems \ref{thm:degree-s on one side} and \ref{thm:cycle} fairly easily. 

The proof of Theorem \ref{thm:dependent random choice} uses a novel variant of the dependent random choice method. The rough idea is to choose random vertices $v_2,v_3,\dots,v_r\in V(\cG)$ and take $A$ to be the set of $v_1\in V(\cG)$ such that $v_1v_2\dots v_r\in E(\cG)$. However, we add two major twists to this. Firstly, we only put into $A$ those vertices $v_1$ for which $d(v_2,\dots,v_r)=\max_i d(v_1,\dots,v_{i-1},v_{i+1},\dots,v_r)$. Secondly, the vertices $v_2,\dots,v_r$ are not chosen \emph{uniformly} at random, but with probability proportional to $1/d(v_2,\dots,v_r)$.

\begin{proof}[Proof of Theorem \ref{thm:dependent random choice}]
Choose $C$ such that $C\geq 4s$ and $\frac{C^{s-1}((r-1)!)^s}{2^{s}r!s^s}-r^2\geq 1$. Since $s\geq 2$, $C$ can be chosen to be independent from $r$. Let $\cG$ be an $n$-vertex $r$-uniform hypergraph with $e(\cG)\geq C\alpha^{\frac{1}{s-1}}t^{\frac{1}{s-1}}n^{r-\frac{1}{s-1}}$. Let
$$D=\frac{C}{2}\alpha^{\frac{1}{s-1}}t^{\frac{1}{s-1}}n^{1-\frac{1}{s-1}}.$$
By successively deleting all edges containing a set of size $r-1$ which lies in less than $D$ edges, we obtain a subhypergraph $\cG'$ (on the same vertex set) with $e(\cG')\geq e(\cG) - n^{r-1}D \ge e(\cG)/2$ such that for every set $S\subset V(\cG)$ of size $r-1$, we have either $d_{\cG'}(S)=0$ or $d_{\cG'}(S)\geq D$. For the rest of the proof, we let $d(S):=d_{\cG'}(S)$ for every set $S$ of size $r-1$. 

For distinct vertices $v_2,\dots,v_{r}\in V(\cG')$, let $$A_{v_2,\dots,v_{r}}=\{v_1\in V(\cG'): v_1v_2\dots v_{r}\in E(\cG'), d(v_2,\dots,v_r)=\max_i(d(v_1,\dots,v_{i-1},v_{i+1},\dots,v_r))\}$$
and let $a(v_2,\dots,v_{r})=|A_{v_2,\dots,v_{r}}|$.

Define $$p=\sum_{\substack{v_2,\dots,v_r: \\ d(v_2,\dots,v_r)>0}} \frac{D}{n^{r-1} d(v_2,\dots,v_r)}.$$
By the definition of $\cG'$, if $d(v_2,\dots,v_r)>0$, then $d(v_2,\dots,v_r)\geq D$, so we have $p\leq 1$. Let us define a random set $\vA\subset V(\cG')$ as follows. With probability $1-p$, we let $\vA=\emptyset$. With probability $p$, we choose a random $(r-1)$-tuple $(\vv_2,\dots,\vv_r)$ of distinct vertices in $\cG'$ in a way that the probability that $\vv_i=v_i$ for every $i$ is $\frac{D}{n^{r-1} d(v_2,\dots,v_r)}$ if $d(v_2,\dots,v_r)>0$ and $0$ otherwise. Set $\vA=A_{\vv_2,\dots,\vv_r}$.

\medskip

\noindent \emph{Claim 1.} $\sum_{v_2,\dots,v_r} a(v_2,\dots,v_r)\geq (r-1)!e(\cG')$.

\medskip

\noindent \emph{Proof of Claim 1.} For any $e\in E(\cG')$, there are at least $(r-1)!$ ordered $r$-tuples $(v_1,\dots,v_r)$ such that $e=v_1v_2\dots v_r$ and $d(v_2,\dots,v_r)=\max_i(d(v_1,\dots,v_{i-1},v_{i+1},\dots,v_r))$. For any such $r$-tuple, we have $v_1\in A_{v_2,\dots,v_r}$.

\medskip

\noindent \emph{Claim 2.} $\mathbb{E}[|\vA|^s]\geq \frac{((r-1)!)^s}{r!}D^s$.

\medskip

\noindent \emph{Proof of Claim 2.} Using H\"older's inequality for three functions with parameters $p_1 = s,$ $p_2 = s, p_3 = s / (s-2),$ we get
$$\left(\sum \frac{a(v_2,\dots,v_r)^s}{d(v_2,\dots,v_r)}\right)\left(\sum d(v_2,\dots,v_r)\right)\left(\sum 1 \right)^{s-2}\geq \left(\sum a(v_2,\dots,v_r)\right)^s,$$
where each sum is over all $(r-1)$-tuples of distinct vertices $(v_2,\dots,v_r)$ with $d(v_2,\dots,v_r)>0$. 
Hence,
$$\sum \frac{a(v_2,\dots,v_r)^s}{d(v_2,\dots,v_r)}\geq \frac{\left(\sum a(v_2,\dots,v_r)\right)^s}{n^{(r-1)(s-2)}\sum d(v_2,\dots,v_r)},$$
so
\begin{align*}
    \mathbb{E}[|\vA|^s]
    &=\sum \frac{D}{n^{r-1} d(v_2,\dots,v_r)}a(v_2,\dots,v_r)^s\geq \frac{D}{n^{(r-1)(s-1)}} \frac{(\sum a(v_2,\dots,v_r))^s}{\sum d(v_2,\dots,v_r)} \\ &\geq \frac{D}{n^{(r-1)(s-1)}}\frac{((r-1)!e(\cG'))^s}{r!e(\cG')}= \frac{((r-1)!)^s D}{r! n^{(r-1)(s-1)}}e(\cG')^{s-1} \\
    &\geq \frac{((r-1)!)^s D}{r! n^{(r-1)(s-1)}}(Dn^{r-1})^{s-1}=\frac{((r-1)!)^s}{r!}D^s,
\end{align*}
where the second inequality used Claim 1.

\medskip

\noindent \emph{Claim 3.} Let $u_1$, $u_2$, \dots, $u_s$, $v_2$, \dots, $v_{r-1}$ be distinct vertices in $\cG'$. Then the probability that $u_1,u_2,\dots,u_s\in \vA$ and $\vv_i=v_i$ for all $2\leq i\leq r-1$ is at most $D/n^{r-1}$.

\medskip

\noindent \emph{Proof of Claim 3.} Assume that $v_r\in V(\cG')$ is such that $u_j\in A_{v_2,\dots,v_r}$ for each $j\in [s]$. Then in particular $u_1v_2v_3\dots v_r\in E(\cG')$ and $d(v_2,v_3,\dots,v_r)\geq d(u_1,v_2,\dots,v_{r-1})$. Clearly, there are at most $d(u_1,v_2,\dots,v_{r-1})$ choices for $v_r$ satisfying these two properties, and for each such choice, the probability that $\vv_i=v_i$ for all $2\leq i\leq r$ is $\frac{D}{n^{r-1} d(v_2,v_3,\dots,v_r)}\leq \frac{D}{n^{r-1} d(u_1,v_2,\dots,v_{r-1})}$. Hence, summing over all possibilities for $v_r$ proves the claim.

\medskip

\noindent \emph{Claim 4.} Let $u_1$, $u_2$, \dots, $u_s$ and $v_2$ be distinct vertices in $\cG'$. Then the probability that $u_1,\dots,u_s\in \vA$ and $\vv_2=v_2$ is at most $D/n^2$.

\medskip

\noindent \emph{Proof of Claim 4.} This follows from Claim 3 and the union bound over all choices for $v_3,\dots,v_{r-1}$.

\medskip

\noindent \emph{Claim 5.} Suppose that $u_1$, $u_2$, \dots, $u_s$ are distinct vertices in $\cG'$ such that $\{u_1,\dots,u_s\}$ is not $t$-rich in $\cG'$. Then the probability that $u_j\in \vA$ for every $j\in [s]$ is at most $\frac{(r-1)^2(t-1)D}{n^2}$.

\medskip

\noindent \emph{Proof of Claim 5.} Since $\{u_1,\dots,u_s\}$ is not $t$-rich, the common intersection of the link hypergraphs of $u_1, \dots, u_s$ does not contain a matching of size $t$. Hence, there is a set $T\subset V(\cG')$ of size at most $(r-1)(t-1)$ with the property that whenever $u_jv_2\dots v_r\in E(\cG')$ for every $j\in [s]$, we have $v_i\in T$ for some $2\leq i\leq r$. Therefore, if $u_j\in \vA$ for every $j\in [s]$, then $\vv_i\in T$ for some $2\leq i\leq r$. By Claim 4, the probability that $u_j\in \vA$ for every $j\in [s]$ and $\vv_2\in T$ is at most $|T|D/n^2$. By symmetry, for any fixed $2\leq i\leq r$, the probability that $u_j\in \vA$ for every $j\in [s]$ and $\vv_i\in T$ is also at most $|T|D/n^2$. The claim follows.

\medskip

Let $\vb$ be the number of sets of size $s$ in $\vA$ which are not $t$-rich. It follows from Claim~5 that $\mathbb{E}[\vb]\leq \binom{n}{s}\cdot \frac{(r-1)^2 (t-1) D}{n^2}\leq r^2 t D n^{s-2}$. By Claim 2 and since $D\geq 4s$, $\mathbb{E}\left[|\vA|^s\mathbbm{1}(|\vA|\geq s)\right]\geq \mathbb{E}\left[|\vA|^s\right]-s^s\geq \frac{((r-1)!)^s}{r!}D^s-s^s\geq \frac{((r-1)!)^s}{2r!}D^s$, so using that $\binom{x}{s} \ge (x/s)^s$ for $x \ge s,$ we have 
$$\mathbb{E}\left[\binom{|\vA|}{s}\right]\geq \mathbb{E}\left[\binom{|\vA|}{s}\mathbbm{1}(|\vA|\geq s)\right]\geq \mathbb{E}[|\vA|^s \mathbbm{1}(|\vA|\geq s)]/s^s\geq \frac{((r-1)!)^s}{r!2s^s}D^s.$$
Hence,
\begin{align*}
    \mathbb{E}\left[\binom{|\vA|}{s}-\alpha \vb\right]
    &\geq \frac{((r-1)!)^s}{r!2s^s}D^s-\alpha r^2 t D n^{s-2}=\left(\frac{((r-1)!)^s}{r!2s^s}D^{s-1}-\alpha r^2 t n^{s-2}\right)D \\
    &=\left(\frac{C^{s-1}((r-1)!)^s}{2^{s}r!s^s}-r^2\right)\alpha t n^{s-2} D\geq \alpha t n^{s-2} D \geq \alpha^{\frac{s}{s-1}} t^{\frac{s}{s-1}} n^{s-\frac{s}{s-1}}.
\end{align*}

It follows that there is an outcome for which $\binom{|\vA|}{s}-\alpha \vb\geq \alpha^{\frac{s}{s-1}} t^{\frac{s}{s-1}} n^{s-\frac{s}{s-1}}$. Then $|\vA|\geq s$, $|\vA|\geq \alpha^{\frac{1}{s-1}}t^{\frac{1}{s-1}}n^{1-\frac{1}{s-1}}$ and the proportion of $t$-rich sets of size $s$ in $\vA$ is at least $1-\alpha^{-1}$, as desired.
\end{proof}

It is now not hard to deduce Theorem \ref{thm:degree-s on one side} and Theorem \ref{thm:cycle}.

\begin{proof}[Proof of Theorem \ref{thm:degree-s on one side}]
We may assume that there is a vertex in $Y$ of degree exactly $s$ in $G$ (else, we can replace $s$ by a smaller number). Let $t=|V(G_{X,Y}^{(r)})|$ and let $\alpha=t^s$. Let $C$ be the constant provided by Theorem \ref{thm:dependent random choice} and let $C'=C\alpha^{\frac{1}{s-1}}t^{\frac{1}{s-1}}$. Note that $C'$ is a constant that depends only on $G$ and $r$.

Let $\cG$ be an $n$-vertex $r$-uniform hypergraph with at least $C'n^{r-\frac{1}{s-1}}$ edges. By Theorem~\ref{thm:dependent random choice}, there is a set $A\subset V(\cG)$ of size at least $\alpha^{\frac{1}{s-1}}t^{\frac{1}{s-1}}n^{1-\frac{1}{s-1}}$ such that the proportion of $t$-rich sets in $A$ is at least $1-\alpha^{-1}$. Note that then $|A|\geq t\geq |X|$. Moreover, the proportion of $t$-rich sets in $A$ is greater than $1-\binom{|X|}{s}^{-1}$, so $A$ has a subset $A'$ of size $|X|$ in which all $s$-sets are $t$-rich. This implies that $\cG$ contains $G_{X,Y}^{(r)}$ as a subgraph. 
Indeed, using that $t=|V(G_{X,Y}^{(r)})|$, we can construct a copy of $G_{X,Y}^{(r)}$ by embedding $X$ arbitrarily into $A'$ and then embedding the sets $Y_1,\dots,Y_m$ from the definition of $G_{X,Y}^{(r)}$ greedily one by one. 
\end{proof}

\begin{proof}[Proof of Theorem \ref{thm:cycle}]
The lower bound was justified in the paragraph before the statement of the theorem, so it is enough to prove the upper bound. Let $C$ be the constant provided by Theorem~\ref{thm:dependent random choice} with $s=2$, and let $\cG$ be an $n$-vertex $r$-uniform hypergraph with at least $100Crtn^{r-1}$ edges. By Theorem \ref{thm:dependent random choice} applied with $\alpha=100$, $s=2$ and $rt$ in place of $t$, there is a set $A\subset V(\cG)$ of size at least $100rt$ such that the proportion of $rt$-rich sets of size $2$ in $A$ is at least $99/100$.

We claim that there is a set $A'\subset A$ of size at least $4t$ such that for each $u\in A'$, the number of $v\in A'$ for which $\{u,v\}$ is $rt$-rich is at least $3|A'|/4$. Indeed, let $A_0=A$ and, recursively for every $i$:
\begin{itemize}
    \item if there is some $u\in A_i$ such that the number of vertices $v\in A_i$ for which $\{u,v\}$ is $rt$-rich is less than $3|A_i|/4$, then choose such a vertex and let $A_{i+1}=A_i\setminus \{u\}$, 
    \item else terminate the process and let $A'=A_i$.
\end{itemize}
Clearly, we obtain a set $A'$ such that for each $u\in A'$, the number of $v\in A'$ for which $\{u,v\}$ is $rt$-rich is at least $3|A'|/4$; we just need to show that $|A'|\geq 4t$. If $|A'|<4t$, then we have deleted at least $|A|/2$ vertices which implies that there were at least $\frac{|A|}{2}\cdot (|A|/8-1)>\frac{1}{100}\binom{|A|}{2}$ pairs in $A$ which are not $rt$-rich. This is a contradiction, so indeed $|A'|\geq 4t$.

\sloppy Observe that for any $v, v' \in A'$, there are at least $|A'| / 2 \ge 2t$ vertices $u\in A'$ such that the pairs $\{u, v\}$ and $\{u, v'\}$ are $rt$-rich. We can now greedily find distinct vertices $x_1,x_2,\dots,x_t$ in $A'$ such that $\{x_1,x_2\}$, $\{x_2,x_3\}$, \dots, $\{x_t,x_1\}$ are $rt$-rich pairs. Since $|V(C_{2t}^{(r)})| = rt$, we can greedily find pairwise disjoint sets $Y_1, \dots, Y_t$ in $V(\cG)\setminus \{x_1,\dots,x_t\}$ such that $\{x_i\} \cup Y_i, \{x_{i+1}\} \cup Y_i \in e(\cG)$ for all $i \in [t]$, where we let $x_{t+1} = x_1$. Hence, $\cG$ contains $C_{2t}^{(r)}$ as a subgraph, completing the proof.
\end{proof}

\section{Lower bounds} \label{sec:lower}

In this section we prove Theorem \ref{thm:right dependence on t}. The key ingredient is the following lemma.

\begin{lemma} \label{lem:general lower bound}
    Let $A$ and $B$ be two disjoint sets of size $n$. Assume that there exist pairwise edge-disjoint bipartite graphs $G_1,G_2,\dots,G_m$ with parts $A$ and $B$ such that for any distinct vertices $x_1,x_2,\dots,x_s\in A\cup B$, there are fewer than $t$ vertices $y\in A\cup B$ for which there exists $i\in [m]$ (that may depend on $y$) with $x_1y,x_2y,\dots,x_sy\in E(G_i)$. Let $e=\sum_{i=1}^m e(G_i)$.
    \nolinebreak Then $$\ex(2kn,K_{s,t}^{(2k)})\geq e^k/m.$$
\end{lemma}

\begin{proof}
    Let $X_1,X_2,\dots,X_{2k}$ be pairwise disjoint sets of size $n$. For every $1\leq p\leq m$, we define a $2k$-partite $2k$-uniform hypergraph $\cG(p)$ with parts $X_1,X_2,\dots,X_{2k}$ as follows. For $x_1\in X_1,\dots,x_{2k}\in X_{2k}$, we let $x_1x_2\dots x_{2k}$ be a hyperedge in $\cG(p)$ if and only if there exist $1\leq i_1,i_2,\dots,i_k\leq m$ such that $i_1+\dots+i_k \equiv p \mod m$ and for each $1\leq \ell\leq k$, we have $x_{2\ell-1}x_{2\ell}\in E(G_{i_\ell})$, where $X_{2\ell-1}$ is identified with $A$ and $X_{2\ell}$ is identified with $B$. Now clearly, $|\bigcup_{p=1}^m E(\cG(p))|=|\bigcup_{i=1}^m E(G_i)|^k=e^k$. Hence, there exists some $p$ for which $e(\cG(p))\geq e^k/m$.
    
    It is therefore sufficient to prove that $\cG(p)$ is $K_{s,t}^{(2k)}$-free for every $p$.
    Suppose otherwise. By symmetry, we may assume that there are distinct vertices $x_{1,1},\dots,x_{1,s}\in X_1$, $x_{\alpha,\beta}\in X_\alpha$ for all $2\leq \alpha\leq 2k$ and $1\leq \beta\leq t$ such that $x_{1,i}x_{2,\beta}x_{3,\beta}\dots x_{2k,\beta}\in E(\cG)$ for each $1\leq i\leq s$ and $1\leq \beta\leq t$. Clearly, for each $2\leq \ell\leq k$ and $1\leq \beta \leq t$, there is a unique $i_{\ell}(\beta)\in [m]$ such that $x_{2\ell-1,\beta}x_{2\ell,\beta}\in E(G_{i_{\ell}(\beta)})$. Moreover, for any $1\leq j\leq s$ and $1\leq \beta \leq t$ there is a unique $i_1(j,\beta)\in [m]$ such that $x_{1,j}x_{2,\beta}\in E(G_{i_1(j,\beta)})$.
    
    By the definition of $\cG(p)$, for any $1\leq j\leq s$ and $1\leq \beta\leq t$, $i_1(j,\beta)+i_2(\beta)+\dots+i_k(\beta)\equiv p \mod m$. Hence, $i_1(1,\beta)=i_1(2,\beta)=\dots =i_1(s,\beta)$. Then for every $1\leq \beta \leq t$, there is some $i\in [m]$ such that $x_{1,1}x_{2,\beta},x_{1,2}x_{2,\beta},\dots,x_{1,s}x_{2,\beta}\in E(G_i)$ (namely $i=i_1(1,\beta)=i_1(2,\beta)=\dots =i_1(s,\beta)$). By assumption, the vertices $x_{2,\beta}, 1 \le \beta \le t$ are all distinct which contradicts the properties of the graphs $G_i$.
\end{proof}

We now want to show that for $m\approx (n/t)^{\frac{1}{s-1}}$, one can almost completely cover the edge set of $K_{n,n}$ with pairwise edge-disjoint graphs $G_1,\dots,G_m$ satisfying the property described in Lemma \ref{lem:general lower bound}. The bound in Lemma \ref{lem:general lower bound} will then give $\ex(2kn,K_{s,t}^{(2k)})\gtrapprox t^{\frac{1}{s-1}}n^{2k-\frac{1}{s-1}}$.

The following lemma provides a suitable collection of subgraphs under some mild divisibility conditions.

\begin{lemma} \label{lem:good partition exists}
    Let $s\geq 2$ and $h$ be positive integers and let $p$ be a prime congruent to $1$ modulo~$h$. Let $m=(p-1)/h$. Then there are pairwise edge-disjoint bipartite graphs $G_1,\dots,G_{m}$ with the same parts $A$ and $B$ such that $|A|=|B|=p^{s-1}$, $|\bigcup_{i=1}^{m} E(G_i)|=p^{2s-2}-p^{s-1}$ and for any distinct vertices $x_1,\dots,x_s\in A\cup B$ there are at most $h^{s-1}(s-1)!$ vertices $y\in A\cup B$ for which there exists $i\in [m]$ with $x_1y,x_2y,\dots,x_sy\in E(G_i)$.
\end{lemma}

In the proof, we make use of the following result of Koll\'ar, R\'onyai and Szab\'o which was used to obtain their celebrated lower bound for the Tur\'an number of complete bipartite graphs; see also \cite{ALON1999280} for a refinement.

\begin{lemma}[{\cite[Theorem 3.3]{KRS96}}] \label{lem:KRS}
    Let $K$ be a field and let $a_{i,j},b_i\in K$ for $1\leq i,j\leq t$ such that $a_{i,j_1}\neq a_{i,j_2}$ for any $i\in [t]$ and $j_1\neq j_2$. Then the system of equations
    \begin{gather*}
        (z_1-a_{1,1})(z_2-a_{2,1})\dots (z_t-a_{t,1})=b_1, \\
        (z_1-a_{1,2})(z_2-a_{2,2})\dots (z_t-a_{t,2})=b_2, \\
         \vdots \\
        (z_1-a_{1,t})(z_2-a_{2,t})\dots (z_t-a_{t,t})=b_t
    \end{gather*}
    has at most $t!$ solutions $(z_1,z_2,\dots,z_t)\in K^t$.
\end{lemma}

\begin{proof}[Proof of Lemma \ref{lem:good partition exists}]
    Let $A$ and $B$ be disjoint copies of the field $\mathbb{F}_{p^{s-1}}$. Let $H$ be a subgroup of $\mathbb{F}_p^{\times}$ of order $h$, where $\mathbb{F}_p^{\times}$ denotes the multiplicative group of $\mathbb{F}_p$. Let $S_1,S_2,\dots,S_m$ be the cosets of $H$ in $\mathbb{F}_p^{\times}$.
    
    Recall that the norm map $N \colon \mathbb{F}_{p^{s-1}}\rightarrow \mathbb{F}_p$ is defined as $N(x) = x \cdot x^p \cdot x^{p^2} \cdots x^{p^{s-2}}$ and note that $N(xy) = N(x)N(y)$ for any $x, y \in \mathbb{F}_{p^{s-1}}$. For $x\in A$, $y\in B$ and $i\in [m]$, let $xy$ be an edge in $G_i$ if and only if $N(x+y)\in S_i$. Since $S_1,\dots,S_m$ partition $\mathbb{F}_p^{\times}$ and $N(z)=0$ if and only if $z=0$, it follows that $G_1,G_2,\dots,G_m$ are pairwise edge-disjoint and $\bigcup_{i=1}^{m} E(G_i)=(A\times B)\setminus \{(x,-x):x\in \mathbb{F}_{p^{s-1}}\}$. Hence, $|\bigcup_{i=1}^{m} E(G_i)|=p^{2s-2}-p^{s-1}$.
    
    We are left to show that for any distinct vertices $x_1,\dots,x_s\in A\cup B$ there are at most $h^{s-1}(s-1)!$ vertices $y\in A\cup B$ for which there exists $i\in [m]$ with $x_1y,x_2y,\dots,x_sy\in E(G_i)$. We may assume without loss of generality that $x_j\in A$ for each $j\in [s]$. Suppose that for some $y\in B$ there is $i\in [m]$ with $x_1y,x_2y,\dots,x_sy\in E(G_i)$. This means that $N(x_j+y)\in S_i$ holds for each $j\in [s]$. Then $N(\frac{x_j+y}{x_s+y})=N(x_j+y)/N(x_s+y)\in H$ for each $j\in [s-1]$.
    
    \medskip
    
    \noindent \emph{Claim.} Let $x_1,\dots,x_s$ be distinct elements of $\mathbb{F}_{p^{s-1}}$ and let $\lambda_1,\dots,\lambda_{s-1}\in H$. Then there are at most $(s-1)!$ elements $y\in \mathbb{F}_{p^{s-1}}$ such that $N(\frac{x_j+y}{x_s+y})=\lambda_j$ for each $j\in [s-1]$.
    
    \medskip
    
    Since there are $h^{s-1}$ ways to choose the possible values of $N(\frac{x_j+y}{x_s+y})$ for $j\in [s-1]$ from $H$, the claim implies the lemma.
    
    \medskip
    
    \noindent \emph{Proof of Claim.} Note that $N(\frac{x_j+y}{x_s+y})=\lambda_j$ is equivalent to $N(\frac{1}{x_s+y}+\frac{1}{x_j-x_s})=\lambda_j/N(x_j-x_s)$. Setting $z=\frac{1}{x_s+y}$, $a_j=\frac{1}{x_j-x_s}$ and $b_j=\lambda_j/N(x_j-x_s)$, the problem is reduced to counting the number of solutions to the system of equations
    \begin{align} \label{eqn:norm}
    \begin{split}
        N(z+a_1)&=b_1, \\
        N(z+a_2)&=b_2, \\
        &\vdots \\
        N(z+a_{s-1})&=b_{s-1}
    \end{split}
    \end{align}
    in the variable $z$. Since $N(z+a_j)=(z+a_j)(z^p+a_j^p)\dots (z^{p^{s-2}}+a_j^{p^{s-2}})$, we can apply Lemma~\ref{lem:KRS} (with $K=\mathbb{F}_{p^{s-1}}$, $t=s-1$, $a_{i,j}=-a_j^{p^{i-1}}$, $z_i=z^{p^{i-1}}$) to see that (\ref{eqn:norm}) has at most $(s-1)!$ solutions for $z$, completing the proof of the claim.
\end{proof}

\begin{remark}
    One can prove a variant of Lemma \ref{lem:good partition exists} using the random algebraic method of Bukh (see \cite{Bukh15} for a detailed example of how this method is applied). More precisely, one can take a uniformly random polynomial $f:\mathbb{F}_p^{2s-2}\rightarrow \mathbb{F}_p$ of a given (large) degree and set $E(G_i)=\{xy: x,y\in \mathbb{F}_p^{s-1}, f(x,y)\in S_i\}$ for all $i\in [m]$, where $S_i$ are defined as in the proof of Lemma \ref{lem:good partition exists}. The proof above uses essentially the same construction for the explicit choice $f(x,y)=N(x+y)$ (with the minor difference that the parts there are identified with $\mathbb{F}_{p^{s-1}}$ rather than $\mathbb{F}_p^{s-1}$).
\end{remark}

We are now in a position to prove Theorem \ref{thm:right dependence on t}.

\begin{proof}[Proof of Theorem \ref{thm:right dependence on t}]
    Let $h=\lfloor ((t-1)/(s-1)!)^{\frac{1}{s-1}}\rfloor \ge 1$. Choose a prime $p$ such that $p\equiv 1 \mod h$ and $\frac{1}{2}(\frac{n}{2k})^{\frac{1}{s-1}}\leq p\leq (\frac{n}{2k})^{\frac{1}{s-1}}$ (since $n$ is sufficiently large, such a prime exists by the prime number theorem for arithmetic progressions).  
    Let $m=(p-1)/h$. Note that $h^{s-1}(s-1)!\leq t-1$. By the existence of the bipartite graphs provided by Lemma \ref{lem:good partition exists} and by Lemma \ref{lem:general lower bound}, we get
    $$\ex(2kp^{s-1},K_{s,t}^{(2k)})\geq (p^{2s-2}-p^{s-1})^k/m\geq 2^{-k}p^{2(s-1)k}/m\geq 2^{-k}hp^{2(s-1)k-1}\geq ct^{\frac{1}{s-1}}n^{2k-\frac{1}{s-1}}$$
    for some positive constant $c=c(k,s)$. Since $2kp^{s-1}\leq n$, this completes the proof.
\end{proof}

\section{An improved upper bound for $r=3$} \label{sec:improved upper}

In this section we prove Theorem \ref{thm:eps improvement}. The following definition will be crucial in the proof. Here and in the rest of this section, we will ignore floor and ceiling signs whenever doing so does not make a substantial difference.

\begin{definition}
    Let $s\geq 3$ be an integer and let $\cG$ be a 3-uniform 3-partite hypergraph with parts $X$, $Y$ and $Z$ of size $n$ each. We call a vertex $z\in Z$ $s$-\emph{nice} in $\cG$ if there exist partitions $X=X_1\cup \dots \cup X_{n^{\frac{1}{s-1}}}$ and $Y=Y_1\cup \dots \cup Y_{n^{\frac{1}{s-1}}}$ into sets of size $n^{1-\frac{1}{s-1}}$ such that if $xyz\in E(\cG)$ for some $x\in X_i$, then $y\in Y_i$. We define $s$-nice vertices in $X$ and $Y$ analogously.
\end{definition}

Observe that if some vertex $z$ is $s$-nice in $\cG$, then it is also $s$-nice in any subhypergraph of $\cG$.

The proof of Theorem \ref{thm:eps improvement} will consist of two main steps. First, we prove the following structural result which states that (under a mild condition on the maximum degree) if a $K_{s,t}^{(3)}$-free hypergraph has close to $n^{3-\frac{1}{s-1}}$ edges, then it contains a subgraph with a similar number of edges in which all vertices in two of the parts are nice.

\begin{lemma} \label{lem:find structure}
    Let $s\geq 3$, let $t$ be a positive integer, let $\eps>0$ and let $n$ be sufficiently large. Let $\cG$ be a $K_{s,t}^{(3)}$-free $3$-uniform $3$-partite hypergraph with parts $X$, $Y$ and $Z$ of size $n$ each. Assume that $e(\cG)\geq n^{3-\frac{1}{s-1}-\eps}$ and that every pair of vertices belongs to at most $n^{1-\frac{1}{s-1}+\eps}$ hyperedges. Then $\cG$ has a subhypergraph $\cH$ (on the same vertex set) such that $e(\cH)\geq n^{3-\frac{1}{s-1}-225s^4\eps}$ and every vertex in $X\cup Y$, $Y\cup Z$ or $Z\cup X$ is $s$-nice in $\cH$.
\end{lemma}

The second step is showing that (again under some mild conditions on the degrees) such a structured hypergraph must contain $K_{s,t}^{(3)}$.

\begin{lemma} \label{lem:structure enough}
   Let $s\geq 3$, let $t$ be a positive integer, let $0<\eps<\frac{1}{4s+6}$ and let $n$ be sufficiently large. Let $\cG$ be a $3$-uniform $3$-partite hypergraph with parts $X$, $Y$ and $Z$ of size $n$ each. Assume that any pair of vertices in $\cG$ is contained in either $0$ or in at least $n^{1-\frac{1}{s-1}-\eps}$ hyperedges, but every pair of vertices is in at most $n^{1-\frac{1}{s-1}+\eps}$ hyperedges. Assume that $xyz\in E(\cG)$ for some $x\in X$, $y\in Y$ and $z\in Z$ and that every vertex in $Z\cup \{x\}$ is $s$-nice in $\cG$. Then $\cG$ contains a copy of $K_{s,t}^{(3)}$.
\end{lemma}

We will give the proof of Lemmas \ref{lem:find structure} and \ref{lem:structure enough} in Subsections \ref{subsec:structure} and \ref{subsec:enough}, respectively. Now let us see how these lemmas imply Theorem \ref{thm:eps improvement}. The last ingredient is a lemma that shows that we can assume that no pair of vertices belongs to many hyperedges.

\begin{lemma} \label{lem:max degree}
   Let $s\geq 3$, let $t$ be a positive integer, let $0<\eps<1/2$, let $n$ be sufficiently large and let $\cG$ be a $K_{s,t}^{(3)}$-free $3$-uniform hypergraph with $3n$ vertices and at least $n^{3-\frac{1}{s-1}-\eps}$ hyperedges. Then $\cG$ has a $3$-partite subgraph $\cH$ with parts of size $n$ such that $e(\cH)\geq n^{3-\frac{1}{s-1}-2\eps}$ and any pair of vertices belongs to at most $n^{1-\frac{1}{s-1}+2\eps}$ hyperedges in $\cH$.
\end{lemma}

\begin{proof}
    Clearly $\cG$ has a $3$-partite subgraph $\cG'$, with parts $X,Y,Z$ of size $n$ each, such that $e(\cG')\geq \frac{2}{9}e(\cG)$. For each $e=xyz\in E(\cG')$, let $\lambda(e)=\max(d_{\cG'}(x,y),d_{\cG'}(y,z),d_{\cG'}(z,x))$. Now there is a positive integer $1\leq b\leq \log_2 n$ such that $\cG'$ has at least $e(\cG')/\log_2 n$ edges $e$ with $2^{b-1}\leq \lambda(e)\leq 2^{b}$. By symmetry, we may assume that there are at least $n^{3-\frac{1}{s-1}-\eps-o(1)}$ triples $(x,y,z)\in X\times Y\times Z$ such that $xyz\in E(\cG')$, $d_{\cG'}(x,y)\geq 2^{b-1}$ and $d_{\cG'}(x,y),d_{\cG'}(y,z),d_{\cG'}(z,x)\leq 2^{b}$. Let $\cH$ be the subgraph of $\cG'$ consisting of precisely these edges $xyz$. Clearly, $2^{b}\geq n^{1-\frac{1}{s-1}-\eps-o(1)}\geq \omega(1)$, so there are at least $n^{3-\frac{1}{s-1}-\eps-o(1)}(2^{b-1})^{s-1}$ many $(s+2)$-tuples $(x,y,z_1,\dots,z_s)\in X\times Y\times Z^s$ of distinct vertices such that $xyz_1\in E(\cH)$ and $xyz_i\in E(\cG')$ for each $i\in [s]$. Write $\mathcal{A}$ for the set of these tuples. By the pigeon hole principle, we can choose some $(z_1,\dots,z_s)\in Z^s$ which features at least $n^{-s}n^{3-\frac{1}{s-1}-\eps-o(1)}(2^{b-1})^{s-1}$ many times in $\mathcal{A}$. Since $\cG'$ does not contain $K_{s,t}^{(3)}$ as a subgraph, there is a set $T\subset X\cup Y$ of size at most $2(t-1)$ such that if $(x,y,z_1,\dots,z_s)\in \mathcal{A}$, then $x\in T$ or $y\in T$. By symmetry, we may therefore assume that for some $y_0\in T \cap Y$ there are at least $\frac{1}{2t-2}n^{-s}n^{3-\frac{1}{s-1}-\eps-o(1)}(2^{b-1})^{s-1}$ vertices $x\in X$ such that $(x,y_0,z_1,\dots,z_s)\in \mathcal{A}$. In particular, $d_{\cG'}(y_0,z_1)\geq \frac{1}{2t-2}n^{-s}n^{3-\frac{1}{s-1}-\eps-o(1)}(2^{b-1})^{s-1}$. On the other hand, $d_{\cG'}(y_0,z_1)\leq 2^b$ by the definition of $\mathcal{A}$. Hence,
    $$\frac{1}{2t-2}n^{-s}n^{3-\frac{1}{s-1}-\eps-o(1)}(2^{b-1})^{s-1}\leq 2^b,$$
    so $2^b\leq n^{1-\frac{1}{s-1}+\frac{\eps}{s-2}+o(1)}\leq n^{1-\frac{1}{s-1}+2\eps}$. This implies that every pair of vertices belongs to at most $n^{1-\frac{1}{s-1}+2\eps}$ hyperedges in $\cH$.
\end{proof}

We can now prove Theorem \ref{thm:eps improvement}.

\begin{proof}[Proof of Theorem \ref{thm:eps improvement}]
    Let $\eps<\frac{1}{900s^4(2s+3)}$, let $n$ be sufficiently large and assume, for the sake of contradiction, that $\cG$ is a $K_{s,t}^{(3)}$-free $3$-uniform hypergraph with $3n$ vertices and at least $n^{3-\frac{1}{s-1}-\eps}$ edges.
    
    By Lemma \ref{lem:max degree}, $\cG$ has a $3$-partite subgraph $\cG'$ with parts $X,Y,Z$ of size $n$ such that $e(\cG')\geq n^{3-\frac{1}{s-1}-2\eps}$ and any pair of vertices belongs to at most $n^{1-\frac{1}{s-1}+2\eps}$ hyperedges in $\cG'$.
    Lemma \ref{lem:find structure} implies that $\cG'$ has a subgraph $\cG''$ (on the same vertex set) such that $e(\cG'')\geq n^{3-\frac{1}{s-1}-450s^4 \eps}$ and every vertex in $X\cup Y$, $Y\cup Z$ or $Z\cup X$ is $s$-nice in $\cG''$. By successively removing edges which contain a pair of vertices lying in less than $D = \frac{1}{10}n^{1-\frac{1}{s-1}-450s^4 \eps}$ edges, we obtain a non-empty subgraph $\cG'''$ (on the same vertex set) in which every pair of vertices belongs to either $0$ or at least $D$ hyperedges. Hence, since $450s^4 \eps<\frac{1}{4s+6}$, Lemma~\ref{lem:structure enough} implies that $\cG'''$ contains $K_{s,t}^{(3)}$ as a subgraph, which is a contradiction.
\end{proof}

\subsection{Finding a structured subgraph in $\cG$} \label{subsec:structure}

In this subsection, we prove Lemma \ref{lem:find structure}. In what follows, for a graph $G$ and vertices $u_1,\dots,u_k\in V(G)$, we write $d_G(u_1,\dots,u_k)$ for the number of common neighbours of $u_1,\dots,u_k$ in $G$. With a slight abuse of notation, for a $3$-uniform hypergraph $\cG$, we still write $d_{\cG}(u,v)$ for the number of hyperedges in $\cG$ containing both $u$ and $v$.

\begin{lemma} \label{lem:partition}
    Let $s\geq 3$, let $\eps>0$ and let $n$ be sufficiently large. Let $G=(X,Y)$ be a bipartite graph on $n+n$ vertices such that $\Delta(G)\leq n^{1-\frac{1}{s-1}+\eps}$ and the number of $s$-tuples $(u_1,\dots,u_s)\in X^s$ with $d_G(u_1,\dots,u_s)\geq n^{1-\frac{1}{s-1}-\eps}$ is at least $n^{s-1-\eps}$. Then there are pairwise disjoint sets $U_1,U_2,\dots,U_k\subset X$ and $V_1,V_2,\dots,V_k\subset Y$ of size $n^{1-\frac{1}{s-1}}$ for some $k\geq n^{\frac{1}{s-1}-(3s+1)\eps}$ such that $G[U_i,V_i]$ has at least $n^{2-\frac{2}{s-1}-4\eps}$ edges for each $i\in [k]$.
\end{lemma}

\begin{proof}
Assume that for some $j<n^{\frac{1}{s-1}-(3s+1)\eps}$ we have already found pairwise disjoint sets $U_1,U_2,\dots,U_j\subset X$ and $V_1,V_2,\dots,V_j\subset Y$ of size $n^{1-\frac{1}{s-1}}$ such that $G[U_i,V_i]$ has at least $n^{2-\frac{2}{s-1}-4\eps}$ edges for each $i\in [j]$. We show how to find the next pair of subsets $U_{j+1}, V_{j+1}$. Write $U=\bigcup_{i\in [j]} U_i$ and $V=\bigcup_{i\in [j]} V_i$. Clearly, $|U|\leq n^{1-(3s+1)\eps}$ and $|V|\leq n^{1-(3s+1)\eps}$. Let $W=\{x\in X: |N_G(x)\cap V|\geq \frac{1}{2}n^{1-\frac{1}{s-1}-\eps}\}$. By double counting the edges between $W$ and $V$, we get $|W|\cdot \frac{1}{2}n^{1-\frac{1}{s-1}-\eps}\leq |V|\Delta(G)$, which implies that $|W|\leq 2n^{1-(3s-1)\eps}$.

For a vertex $u\in X$, let $S_u=\{x\in X: d_G(u,x)\geq n^{1-\frac{1}{s-1}-\eps}\}$. By double counting the edges between $S_u$ and $N_G(u)$, we get $|S_u|n^{1-\frac{1}{s-1}-\eps}\leq \Delta(G)^2$, so $|S_u|\leq n^{1-\frac{1}{s-1}+3\eps}$. It follows that for any $u\in X$, the number of $(u_2,\dots,u_s)\in X^{s-1}$ with $d_G(u,u_2,u_3,\dots,u_s)\geq n^{1-\frac{1}{s-1}-\eps}$ is at most $|S_u|^{s-1}\leq n^{s-2+3(s-1)\eps}$. Therefore, the number of $(u_1,\dots,u_s)\in X^s$ with $d_G(u_1,\dots,u_s)\geq n^{1-\frac{1}{s-1}-\eps}$ such that $u_i\in U\cup W$ for some $i\in [s]$ is at most $s(|U|+\nolinebreak |W|)n^{s-2+3(s-1)\eps}\leq 3sn^{s-1-2\eps}\leq \frac{1}{2}n^{s-1-\eps}$. 
On the other hand, by assumption, the number of $s$-tuples $(u_1,\dots,u_s) \in X^s$ with $d_G(u_1,\dots,u_s)\geq n^{1-\frac{1}{s-1}-\eps}$ is at least $n^{s-1-\eps}$.
Hence, there exists some $u\in X\setminus (U\cup W)$ such that there are at least $\frac{1}{2}n^{s-2-\eps}$ tuples $(u_2,\dots,u_s)\in (X\setminus (U\cup W))^{s-1}$ with $d_G(u,u_2,u_3,\dots,u_s)\geq n^{1-\frac{1}{s-1}-\eps}$. This implies that there are at least $(\frac{1}{2}n^{s-2-\eps})^{\frac{1}{s-1}}\geq \frac{1}{2}n^{1-\frac{1}{s-1}-\eps}$ vertices $x\in X\setminus (U\cup W)$ with $d_G(u,x)\geq n^{1-\frac{1}{s-1}-\eps}$. Since $u\not \in W$, we have $|N_G(u)\cap V|< \frac{1}{2}n^{1-\frac{1}{s-1}-\eps}$, so there are at least $\frac{1}{2}n^{1-\frac{1}{s-1}-\eps}$ vertices $x\in X\setminus (U\cup W)$ with $|N_G(u)\cap N_G(x)\setminus V|\geq \frac{1}{2}n^{1-\frac{1}{s-1}-\eps}$. This means that we can choose a set $U_{j+1}\subset X\setminus U$ of size $n^{1-\frac{1}{s-1}}$ which sends at least $(\frac{1}{2}n^{1-\frac{1}{s-1}-\eps})^2=\frac{1}{4}n^{2-\frac{2}{s-1}-2\eps}$ edges to $N_G(u)\setminus V$. Since $|N_G(u)\setminus V|\leq \Delta(G)$, there exists a set $V_{j+1}\subset Y\setminus V$ of size $n^{1-\frac{1}{s-1}}$ such that the number of edges in $G[U_{j+1},V_{j+1}]$ is at least 
$\frac{1}{4}n^{2-\frac{2}{s-1}-2\eps} \cdot \min\left(1,n^{1-\frac{1}{s-1}}/\Delta(G)\right) \geq   \frac{1}{4}n^{2-\frac{2}{s-1}-3\eps}\geq n^{2-\frac{2}{s-1}-4\eps}$. This completes the proof.
\end{proof}

\begin{lemma} \label{lem:one part nice}
   Let $s\geq 3$, let $t$ be a positive integer, let $\eps>0$ and let $n$ be sufficiently large. Let $\cG$ be a $K_{s,t}^{(3)}$-free $3$-uniform $3$-partite hypergraph with parts $X$, $Y$ and $Z$ of size $n$ each. Assume that $e(\cG)\geq n^{3-\frac{1}{s-1}-\eps}$ and that every pair of vertices belongs to at most $n^{1-\frac{1}{s-1}+\eps}$ hyperedges. Then $\cG$ has a subhypergraph $\cG'$ (on the same vertex set) such that $e(\cG')\geq n^{3-\frac{1}{s-1}-15s^2\eps}$ and either every vertex in $Y$ or every vertex in $Z$ is $s$-nice in $\cG'$.
\end{lemma}

\begin{proof}
    We may assume that $\eps<1/4$, else the conclusion of the lemma holds trivially. Using $e(\cG)\geq n^{3-\frac{1}{s-1}-\eps}$, by convexity there is a set $\mathcal{T}$ of at least $\Omega_s(n^2(n^{1-\frac{1}{s-1}-\eps})^s) = \Omega_s(n^{s+1-\frac{1}{s-1}-s\varepsilon})$ tuples $(x_1,x_2,\dots,x_s,y,z)$ of distinct vertices such that $x_i\in X$, $y\in Y$ and $z\in Z$ and $x_iyz\in E(\cG)$ for every $i$. Since $\cG$ is $K_{s,t}^{(3)}$-free, for any distinct $x_1,\dots,x_s\in X$ there is a set $S\subset Y\cup Z$ of size at most $2t-2$ such that for each $y\in Y$ and $z\in Z$ for which $(x_1,x_2,\dots,x_s,y,z)\in \mathcal{T}$, we have $y\in S$ or $z\in S$. Since every pair of vertices in $\cG$ is in at most $n^{1-\frac{1}{s-1}+\eps}$ hyperedges, it follows that any fixed $(x_1,\dots,x_s)$ extends to at most $(2t-2)n^{1-\frac{1}{s-1}+\eps}$ members of $\mathcal{T}$. Hence, there are at least $\frac{\frac{1}{2}|\mathcal{T}|}{(2t-2)n^{1-\frac{1}{s-1}+\eps}}$ tuples $(x_1,\dots,x_s)$ which extend to at least $\frac{\frac{1}{2}|\mathcal{T}|}{n^s}$ members of $\mathcal{T}$. For each such $(x_1,\dots,x_s)$ there is some $z\in Y\cup Z$ such that $d_{\cG}(\{x_1,\dots,x_s\},z)\geq
    \frac{1}{2t-2} \cdot \frac{\frac{1}{2}|\mathcal{T}|}{n^s} \geq \Omega_{s,t}(n^{1 - \frac{1}{s-1}-s\varepsilon}) \geq
    n^{1-\frac{1}{s-1}-(s+1)\eps}$. Here $d_{\cG}(\{x_1,\dots,x_s\},z)$ denotes the number of vertices $y\in V(\cG)$ such that $x_iyz\in E(\cG)$ holds for all $i\in [s]$. Hence, the number of tuples $(x_1,x_2,\dots,x_s,z)\in X^s\times (Y\cup Z)$ of distinct vertices with $d_{\cG}(\{x_1,\dots,x_s\},z)\geq n^{1-\frac{1}{s-1}-(s+1)\eps}$ is at least $\frac{\frac{1}{2}|\mathcal{T}|}{(2t-2)n^{1-\frac{1}{s-1}+\eps}}\geq \Omega_{s,t}(n^{s-(s+1)\eps})\geq 4n^{s-(s+2)\eps}$. By the symmetry of $Y$ and $Z$, we may assume, without loss of generality, that there are at least $2n^{s-(s+2)\eps}$ tuples $(x_1,x_2,\dots,x_s,z)\in X^s\times Z$ with $d_{\cG}(\{x_1,\dots,x_s\},z)\geq n^{1-\frac{1}{s-1}-(s+1)\eps}$.
    
    For every $z\in Z$, define a bipartite graph $G_z$ with parts $X$ and $Y$ where $xy$ is an edge in $G_z$ if and only if $xyz\in E(\cG)$. Observe that for any $z\in Z$, we have $\Delta(G_z)\leq n^{1-\frac{1}{s-1}+\eps}$. By the previous paragraph,
    \begin{equation}\label{eq:one part nice 1}
    \sum_{z\in Z} \left| \left\{(x_1,\dots,x_s)\in X^s: d_{G_z}(x_1,\dots,x_s)\geq n^{1-\frac{1}{s-1}-(s+1)\eps} \right\}\right|\geq 2n^{s-(s+2)\eps}.
    \end{equation}
    On the other hand, we claim that for any $z\in Z$, we have 
    \begin{equation}\label{eq:one part nice 2}
    \left| \left\{(x_1,\dots,x_s)\in X^s: d_{G_z}(x_1,\dots,x_s)\geq n^{1-\frac{1}{s-1}-(s+1)\eps} \right\}\right|\leq n^{s-1+(s-1)(s+3)\eps}.
    \end{equation}
    Indeed, let $x \in X$ and let $S_x = \{x' \in X : d_{G_z}(x,x') \geq n^{1-\frac{1}{s-1}-(s+1)\varepsilon} \}$. By double counting the edges of $G_z$ between $S_x$ and $N_{G_z}(x)$, we have $|S_x| \cdot n^{1-\frac{1}{s-1}-(s+1)\varepsilon} \leq \Delta(G_z)^2 \leq n^{2-\frac{2}{s-1}+2\varepsilon}$ and so $|S_x| \leq n^{1-\frac{1}{s-1} + (s+3)\varepsilon}$. Hence, the number of $x_2,\dots,x_s \in X$ satisfying $d_{G_z}(x,x_2,\dots,x_s) \geq n^{1-\frac{1}{s-1}-(s+1)\varepsilon}$ is at most $|S_x|^{s-1} \leq n^{s-2 + (s-1)(s+3)\varepsilon}$. Now \eqref{eq:one part nice 2} follows by summing over $x$. 
    
    By \eqref{eq:one part nice 1} and \eqref{eq:one part nice 2}, there are at least $n^{1-(s+2)\eps-(s-1)(s+3)\eps}$ vertices $z\in Z$ for which
    \begin{equation}
        \left| \left\{(x_1,\dots,x_s)\in X^s: d_{G_z}(x_1,\dots,x_s)\geq n^{1-\frac{1}{s-1}-(s+1)\eps} \right\}\right|\geq n^{s-1-(s+2)\eps}. \label{eqn:good vertices}
    \end{equation}
    
    We now define a suitable subhypergraph of $\cG$.
    For any vertex $z$ satisfying (\ref{eqn:good vertices}), Lemma~\ref{lem:partition} (applied for $G_z$ with $(s+2)\eps$ in place of $\eps$) implies that for some $k \ge n^{\frac{1}{s-1}-(3s+1)(s+2)\eps}$ there are disjoint sets $X_1, \dots, X_k \subseteq X$ and $Y_1, \dots, Y_k \subseteq Y$ of size $n^{1-\frac{1}{s-1}}$ such that for all $i \in [k]$, in $G_z$ there are at least $n^{2-\frac{2}{s-1}-4(s+2)\varepsilon}$ edges between $X_i$ and $Y_i$. Among the hyperedges of $\cG$ containing $z$, keep those $xyz$ for which there is $i \in [k]$ such that $x\in X_i$ and $y\in Y_i$. Thus, for each $z$ satisfying (\ref{eqn:good vertices}) we keep at least $k \cdot n^{2-\frac{2}{s-1}-4(s+2)\varepsilon} \geq n^{2-\frac{1}{s-1} - (3s+5)(s+2)\varepsilon}$ edges containing it. For each $z\in Z$ which does not satisfy (\ref{eqn:good vertices}), delete all hyperedges of $\cG$ containing $z$. Call the resulting subhypergraph $\cG'$.
    
    It is clear that every vertex in $Z$ is $s$-nice in $\cG'$. Moreover, $$e(\cG')\geq n^{1-(s+2)\eps-(s-1)(s+3)\eps} \cdot n^{2-\frac{1}{s-1}-(3s+5)(s+2)\eps}\geq n^{3-\frac{1}{s-1}-15s^2\eps},$$ 
    as $s+2+(s-1)(s+3)+(3s+5)(s+2) \leq 15s^2$ for $s \geq 3$. This completes the proof. 
\end{proof}

\begin{proof}[Proof of Lemma \ref{lem:find structure}]
    The lemma follows from two applications of Lemma \ref{lem:one part nice}.
\end{proof}

\subsection{Finding $K_{s,t}^{(3)}$ in the structured subgraph} \label{subsec:enough}

In this subsection, we prove Lemma \ref{lem:structure enough}. In what follows, for a $3$-uniform hypergraph $\cG$ and distinct vertices $x,z\in V(\cG)$, we write $N_{\cG}(x,z)$ for the set of vertices $y\in V(\cG)$ for which $xyz$ is an edge in $\cG$. Let $K_{1,s,1}^{(3)}$ denote the $3$-uniform hypergraph with vertices $x,y_1,\dots,y_s,z$ and edges $xy_iz$, $i = 1,\dots,s$. 
  
\begin{lemma} \label{lem:many K1s1}
   Let $s\geq 3$, let $0<\eps<1/8$ and let $n$ be sufficiently large. Let $\cG$ be a $3$-uniform $3$-partite hypergraph with parts $X$, $Y$ and $Z$ of size $n$ each. Assume that every pair of vertices in $\cG$ is contained in either $0$ or at least $n^{1-\frac{1}{s-1}-\eps}$ hyperedges. Suppose that $xyz\in E(\cG)$ for some $x\in X$, $y\in Y$ and $z\in Z$ and that $z$ is $s$-nice in $\cG$. Then $\cG$ has at least $n^{s-\frac{2}{s-1}-(2s+1)\eps}$ copies of $K_{1,s,1}^{(3)}$ containing $z$ and with the part of size $s$ being a subset of $N_{\cG}(x,z)$.
\end{lemma}

\begin{proof}
    Let $G_z$ be the link graph of $z$.
    Since $z$ is $s$-nice in $\cG$, there are sets $X'\subset X$ and $Y'\subset Y$ of size $n^{1-\frac{1}{s-1}}$ such that $x\in X'$, $y\in Y'$ and whenever $x'y'\in E(G_z)$, then either none or both of $x'\in X'$ and $y'\in Y'$ hold. By the assumption in the lemma, every vertex in $G_z$ has degree either 0 or at least $n^{1-\frac{1}{s-1}-\eps}$.
    
    Let $xy'\in E(G_z)$. Clearly, $y'\in Y'$. Then $y'$ has at least $n^{1-\frac{1}{s-1}-\eps}$ neighbours in $G_z$, all of which must be in $X'$. Hence, there are at least $n^{1-\frac{1}{s-1}-\eps} \cdot |N_{G_z}(x)| = n^{-\eps}|X'||N_{G_z}(x)|$ edges in $G_z$ between $X'$ and $N_{G_z}(x)$. Since $|N_{G_z}(x)|$ is much larger than $n^{\eps}$, by convexity there are $\Omega_s(|X'||N_{G_z}(x)|^sn^{-s\eps})$ copies of $K_{1,s}$ in $G_z$ with the part of size $s$ inside $N_{G_z}(x)$. Since $|X'|=n^{1-\frac{1}{s-1}}$ and $|N_{G_z}(x)|\geq n^{1-\frac{1}{s-1}-\eps}$, the lemma follows.
\end{proof}

\begin{proof}[Proof of Lemma \ref{lem:structure enough}]
    Since $x$ is $s$-nice, there are sets $Y'\subset Y$ and $Z'\subset Z$ of size $n^{1-\frac{1}{s-1}}$ such that $y\in Y'$, $z\in Z'$ and whenever $xy'z'\in E(\cG)$, then we have either none or both of $y'\in Y'$ and $z'\in Z'$.
    Note that there are at least $n^{1-\frac{1}{s-1}-\eps}$ vertices $z'\in Z'$ such that $xyz'\in E(\cG)$, because each pair of vertices is in either $0$ or at least $n^{1-\frac{1}{s-1}-\eps}$ edges, by the assumption of the lemma. 
    For each $z'\in Z'$ with $xyz'\in E(\cG)$, Lemma \ref{lem:many K1s1} gives at least $n^{s-\frac{2}{s-1}-(2s+1)\eps}$ copies of $K_{1,s,1}^{(3)}$ containing $z'$ and with the part of size $s$ being a subset of $N_{\cG}(x,z')$. Since $N_{\cG}(x,z')\subset Y'$ for every such $z'$, it follows that $\cG$ contains at least $n^{1-\frac{1}{s-1}-\eps}\cdot n^{s-\frac{2}{s-1}-(2s+1)\eps}=n^{s+1-\frac{3}{s-1}-(2s+2)\eps}$ copies of $K_{1,s,1}^{(3)}$ with the part of size $s$ being a subset of $Y'$. However, $|Y'|^s=(n^{1-\frac{1}{s-1}})^s = n^{s-1-\frac{1}{s-1}}$, so it follows by the pigeon hole principle that there is a set $S\subset Y'$ of size $s$ which extends to at least $\frac{n^{s+1-\frac{3}{s-1}-(2s+2)\eps}}{n^{s-1-\frac{1}{s-1}}}=n^{2-\frac{2}{s-1}-(2s+2)\eps}$ copies of $K_{1,s,1}^{(3)}$. 
    Let $E$ be the set of pairs $(x',z') \in X \times Z$ with $x'y'z' \in E(\mathcal{G})$ for every $y' \in S$, so $|E| \geq n^{2-\frac{2}{s-1}-(2s+2)\eps}$.
    We claim that $\cG$ contains a copy of $K_{s,t}^{(3)}$ with the part of size $s$ being $S$. If not, then the pairs in $E$ are covered by at most $2t-2$ vertices. But then $|E| \leq (2t-2) \cdot n^{1-\frac{1}{s-1}+\eps}$, as every pair of vertices is in at most $n^{1-\frac{1}{s-1}+\eps}$ hyperedges. Since $2-\frac{2}{s-1}-(2s+2)\eps>1-\frac{1}{s-1}+\eps$, this contradicts $|E| \geq n^{2-\frac{2}{s-1}-(2s+2)\eps}$.
\end{proof}

\section{Concluding remarks} \label{sec:concluding remarks}
\begin{itemize}
    \item The most interesting question arising from the present paper is whether $\ex(n,K_{s,t}^{(r)}) = O_{r,s,t}(n^{r - \frac{1}{s-1} - \varepsilon})$ for $s\geq 3$ and odd $r \geq 5$. Recall that this is true for $r = 3$ (Theorem~\ref{thm:eps improvement}) but false for even $r$ if $t \gg s$ (Theorem \ref{thm:right dependence on t}).
    \item Similarly, it would be interesting to decide whether $\ex(n,K_{2,t}^{(r)})=\Theta_r(tn^{r-1})$ for odd $r\geq 5$. This is true for $r=3$ and every even $r\geq 4$. The upper bound holds for arbitrary $r\geq 3$ (Theorem \ref{thm:r-uniform bipartite}).
    \item Mubayi and Verstra\"ete conjectured that $\ex(n,K_{s,t}^{(3)}) = \Theta_{s,t}(n^{3-2/s})$ for $2 \leq s \leq t$. This remains open for $s \geq 3$.
\end{itemize}

\end{document}